\documentclass[11pt]{amsart}
\usepackage{amsmath}
\usepackage{amssymb}
\usepackage{amsthm}
\usepackage{verbatim}
\usepackage{stmaryrd}
\usepackage{color}
\usepackage[active]{srcltx}

\newtheorem{thm}{Theorem}[section]

\newtheorem{lem}[thm]{Lemma}

\theoremstyle{definition}
\newtheorem{defn}[thm]{Definition}

\theoremstyle{remark}
\newtheorem{rem}[thm]{Remark}

\newcommand{\tmop}[1]{\ensuremath{\operatorname{#1}}}
\newcommand{\strong}[1]{\textbf{#1}}
\newcommand{\PF}{\text{\large{$\pitchfork$}}}

\title{On the set of metrics without local limiting Carleman weights}

\def\R{\mathbb{R}}

\def\ETM{\mathcal{G}^{\widetilde{\mathcal{EW}}}(M)}
\def\SCYTM{\mathcal{G}^{\widetilde{\mathcal{SCY}}}(M)}
\newcommand{\La}{\Lambda}
\def\2L{\Lambda_{\tilde{\gamma}}}
\def\1L{\Lambda_{\gamma}}

\renewcommand{\ge}{\geqslant}

\renewcommand{\geq}{\geqslant}

\begin{document}

\author{Pablo Angulo-Ardoy}
\address{ Department of Mathematics, Universidad Aut\'onoma de Madrid}
\curraddr{}
\email{pablo.angulo@uam.es}
\thanks{The author was supported by research grant ERC 301179}

\begin{abstract}
In the paper \cite{AFGR} it is shown that the set of Riemannian metrics which do not admit global limiting Carleman weights is open and dense, by studying the conformally invariant Weyl and Cotton tensors.
In the paper \cite{LS} it is shown that the set of Riemannian metrics which do not admit local limiting Carleman weights at any point is residual, showing that it contains the set of metrics for which there are no local conformal diffeomorphisms between any distinct open subsets.
This paper is a continuation of \cite{AFGR}, in order to prove that the set of Riemannian metrics which do not admit \emph{local} limiting Carleman weights \emph{at any point} is open and dense.
\end{abstract}

\maketitle
\pagestyle{myheadings}
\markleft{P.ANGULO-ARDOY}


\section{Introduction}

The inverse problem posed by Calder\'on asks for the determination of the conductivity of a medium by imposing different voltages in the boundary of a domain and measuring the induced current at points in the boundary.
It is unknown if the problem can been solved in this generality, but this voltage to current data is known to be enough to determine the conductivity in the interior of the domain under some circumstances.
In \cite{DKSU07} it was shown that a few inverse problems are solvable in a domain if the Riemannian manifold induced from the conductivity coefficients is \emph{admissible}.
The main local restriction is existence of a so called \emph{limiting Carleman weight} (LCW).
Local existence of LCWs admits a nice geometric interpretation:

\begin{thm}[{\cite[Theorem 1.2]{DKSU07}}]\label{DKSU07}
Let $U$ be a simply-connected open subset of the Riemannian manifold $(M,g)$.
Then $(U,g)$ admits a limiting Carleman weight if and only if some conformal multiple of the metric $g$ admits a parallel unit vector field.
\end{thm}

In \cite{LS}, it was shown that the set of Riemannian metrics which do not admit local limiting Carleman weights at any point is residual (for the $C^\infty$ topology), showing that it contains the set of metrics for which there are no local conformal diffeomorphisms between any distinct open subsets.

\begin{thm}[{\cite[Corollary 1.3]{LS}}]
 Let $(U, g)$ be an open submanifold of some compact manifold $(M, g)$ without boundary, having dimension $n\geq 3$.
 There is a residual set of Riemannian metrics on M (for the $C^\infty$ topology) which do not admit limiting Carleman weights near any point of $U$.
\end{thm}

However, it is very hard to check if a manifold admits local conformal diffeomorphisms.
Indeed, until \cite{AFGR} appeared, it was difficult to tell if a given metric admits LCWs or not.
In that paper, it is shown that if a metric on a manifold of dimension $n\geq 4$ admits a global LCW, then its Weyl tensor has the \emph{eigenflag property} at every point.
Thus if, for a given metric of dimension $n\geq 4$, the Weyl tensor at one point does not satisfy the eigenflag property, then the metric cannot admit a global LCW.

\begin{defn}\label{eigenflag}
 Let $W$ be a \emph{Weyl operator} on a vector space $V$.
 We say that  $W$ has the \emph{eigenflag property} if and only if there is a vector $v\in V$ such that $W(v \wedge v^{\perp})\subset v \wedge v^{\perp}$.
 In other words, for any $w_1,w_2,w_3\in v^\perp$, we have $W(v\wedge w_1,w_2\wedge w_3)=0$.
\end{defn}

\begin{thm}[{\cite[Theorem 1.3]{AFGR}}]\label{THM} Let $(M,g)$ be a Riemannian manifold of dimension $n\ge 4$.
Assume that a metric $\tilde{g} \in [g]$ admits a parallel vector field. Then for any $p\in M$, the Weyl tensor at $p$ has the eigenflag property.
\end{thm}

The authors also provide a similar criterion in dimension $n=3$, but this one involves the Cotton-York tensor: if a metric on a manifold of dimension $n=3$ admits a global LCW, then its Cotton-York is singular.

\begin{thm}[{\cite[Theorem 1.6]{AFGR}}]\label{THM2}
  Let $n=3$. If a metric $\tilde{g} \in [g]$ admits a parallel vector field, then for any $p\in M$

  $$\det(CY_p)=0.$$
\end{thm}

This allows the authors to provide many explicit examples of Riemannian manifolds which do not admit local limiting Carleman weights, and to show that the set of Riemannian metrics which do not admit global limiting Carleman weights contains an open and dense set in the $C^2$ topology.

 \begin{thm}[{\cite[Theorem 1.9]{AFGR}}]\label{size}
 Let $(U, g)$ be an open submanifold of some compact manifold
$(M, g)$ without boundary, having dimension $n\geq  3$. The
set of Riemannian metrics on $M$ which do not admit limiting Carleman
weights near one fixed point of $U$ contains an open and dense subset of the set of all metrics,
endowed with the $C^3$ topology for $n=3$, and the $C^2$ topology for $n\geq 4$.
 \end{thm}

In this paper we show that the set of metrics on a compact manifold with boundary which do no admit local LCWs is an open and dense set in the $C^\infty$ topology, building on the techniques in \cite{AFGR}.

We make use of some basic results in \emph{transversality theory}.
Transversality theory is a very general framework that is often useful to prove that the set of certain objects (e.g. Riemannian metrics on a given manifold) for which a certain derived object (e.g. its Weyl tensor) avoids a certain non typical set (e.g. the set of Weyl tensors with the eigenflag property) is generic, and sometimes even open and dense, in a certain topology.

In our specific situation, applying transversality theory is a little bit more involved than usual, since we work with sections of vector bundles instead of plain maps, and we want to study transversality to stratified sets instead of plain smooth manifolds.
We have dealt with this extra difficulties in a general way that may work in other similar situations, and we have tried to make the presentation as self-contained as possible, so that we only depend on the standard reference \cite{Hirsch} and the paper \cite{KoikeShiota}.
We believe that the strategy presented here may work if the metric is known to belong to some special family of metrics, for which the existence of local LCWs is not known.

Each Riemannian metric of dimension at least $4$ has a \emph{Weyl section}, that assigns to a point $p$ the Weyl tensor of the metric at $p$.
This is a section of a suitable vector bundle, \emph{the Weyl bundle}.
The \emph{eigenflag bundle} is the subset of the Weyl bundle consisting of the Weyl tensors with the eigenflag property.
This fiber bundle is no longer a vector bundle, but instead its fibers are semialgebraic subsets of the space of Weyl tensors.
A metric is \emph{eigenflag-transverse} if its Weyl section is transverse to the eigenflag bundle.
These definitions are made precise in definitions \ref{defn: weyl section etc} and \ref{defn: eigenflag bundle etc}.

\begin{thm}\label{size local}
 Let $(U, g)$ be an open submanifold of some compact manifold $(M, g)$ without boundary, having dimension $n\geq  4$.
Then:
\begin{itemize}
 \item The set of Riemannian metrics on $M$ which do not admit limiting Carleman weights near any point of $U$ contains the set of \emph{eigenflag transverse} metrics.
 \item The set of \emph{eigenflag transverse} metrics is an open and dense subset of the set of all metrics, with the strong $C^\infty$ topology.
\end{itemize}
\end{thm}

For manifolds of dimension $3$, we use the Cotton-York tensor instead of the Weyl tensor.
The \emph{Cotton-York section} assigns to a point $p$ the Cotton-York tensor at $p$.
This is a section of the vector bundle of symmetric traceless bilinear operators on the tangent space $S^2_0(T_p M)$.
This vector bundle has a fiber sub-bundle consisting of the operators in $S^2_0(T_p M)$ whose determinant is zero.
The fiber is a singular algebraic set.
A metric is \emph{CY-transverse} if its Cotton-York section is transverse to the singular Cotton-York bundle.
See \ref{defn: cotton york section etc} for the precise definitions.

\begin{thm}\label{size local dim=3}
 Let $(U, g)$ be an open submanifold of some compact manifold $(M, g)$ without boundary, having dimension $n=3$.
Then:
\begin{itemize}
 \item The set of Riemannian metrics on $M$ which do not admit limiting Carleman weights near any point of $U$ contains the set of \emph{CY-transverse} metrics.
 \item The set of \emph{CY-transverse} metrics is an open and dense subset of the set of all metrics, with the strong $C^\infty$ topology.
\end{itemize}
\end{thm}

This paper is actually a companion to \cite{AFGR}, and we refer the reader to that paper for a more detailed introduction.

\textbf{Acknowledgments:} We thank Daniel Faraco and Luis Guijarro, for conversations on this topic.
We also thank an anonymous referee who made a careful reading of a previous version of this document.

\section{Weyl and Cotton-York tensors}
\label{tensors}

Let us define the $(0,4)$ curvature tensor
$$
R(z, u, v, w) = g\left( z, \nabla_u \nabla_v w
                - \nabla_v \nabla_u w - \nabla_{[u, v]} w \right),
$$
whose trace is the Ricci tensor
$$
Ric(u, v) = \sum_i R(u, e_i, v, e_i),
$$
where $e_i$ is any orthonormal frame $\{e_i\}$.
The trace of $Ric$ is the scalar curvature
$$
s = \sum_i Ric(e_i, e_i)
$$
and the Schouten tensor is given by

\begin{equation}\label{Sformula}
S=\frac{1}{n-2}\left( Ric-\frac{1}{2(n-1)}s g\right).
\end{equation}

The symmetries of the curvature tensor at a point $p$ allow to consider it as a symmetric bilinear operator $\rho$ of the space of bivectors $\La^2(T_pM)$:
$$
\rho(x\wedge y, z\wedge t) = R(x,y,z,t).
$$
Here we have used a different letter to distinguish the $(0,4)$ tensor $R$ from the curvature operator $\rho \in S^2(\La^2(T_pM))$, but we will use the same letter $R$ in the rest of the paper, when the context will make clear which one we use.

The \emph{Kulkarni-Nomizu product} $\owedge:S^2(V)\times S^2(V)\rightarrow S^2(\Lambda^2(V))$ of two symmetric $2$-tensors is defined by

\[ (\alpha \owedge \beta)_{ijkl}=\alpha_{ik}\beta_{jl}+\beta_{ik}\alpha_{jl}-\alpha_{il}\beta_{jk}-\beta_{jk}\alpha_{il}. \]

With these definitions, we can define the \emph{Weyl operator} $W$:
\begin{equation}\label{decomposition}
  W= R - S \owedge g.
\end{equation}

There is another way of looking at the curvature and the Weyl operators.
Let $V$ be an euclidean space with an inner product $\langle\cdot\rangle$.
We define the Bianchi map for $V$ $b_V:S^2(\La^2(V))\rightarrow\La^4(V)$:
\begin{equation}\label{eq:bianchi}
b_V(R)(x,y,z,t)
=\frac13 \left( R(x\wedge y, z\wedge t)+
	        R(y\wedge z, x\wedge t)+
	        R(z\wedge x, y\wedge t)
	 \right).
\end{equation}
The first Bianchi identity says the curvature operator $R_p$ is always in the kernel of $b_{T_p M}$, which justifies that $\mathcal{R}(V)=\ker(b_V)$ is called the space of \emph{curvature operators} for $V$.

Note that the Bianchi map is defined without using the inner product of $V$.
The Ricci contraction $r_V: S^2(\Lambda^2(V)) \to S^2(V)$ is the map:
\begin{equation} \label{eq:ricci}
r_V(R)(x,y)=Tr\left[R(x,\cdot,y,\cdot)\right] = \sum_i R(x\wedge e_i,y\wedge e_i)
\end{equation}
for an orthonormal frame $\{e_i\}$.
The \emph{space of Weyl operators} for $V$ is a linear subspace of the space of curvature operators:
\begin{equation}\label{space of Weyl operators}
\mathcal{W}(V)
=ker(b_V)\cap ker(r_V).
\end{equation}

It can be checked easily that $b_V(S\owedge g)=0$ and $r_V(S\owedge g)=Ric=r_V(R)$, hence the Weyl operator at $p$ always lies in $\mathcal{W}(T_p M)$.
Lemma \ref{perturb metric to get any algebraic curvature} and identity (\ref{decomposition}) show that for any element in a neighborhood of $0\in\mathcal{W}(T_p M)$, we can find a metric whose Weyl operator at $p$ is that element.

The Weyl tensor is widely used in conformal geometry, since its $(1,3)$ version is conformally invariant (while the Weyl operator scales by the same factor as the metric).
Furthermore, if a metric on a manifold of dimension $n\geq 4$ has vanishing Weyl tensor, it is conformally flat.

However, the Weyl tensor vanishes in dimension lower than $4$, so in dimension $3$ we will use instead the Cotton tensor and its variant, the Cotton-York tensor.
The Cotton tensor is a $(0,3)$ tensor defined as
\begin{equation}\label{def: Cotton}
C_{ijk}=\left(\nabla_i S\right)_{jk}-\left(\nabla_j S\right)_{ik}.
\end{equation}

where the notation $(\nabla_a S)_{bc}$ stands for $(\nabla_{\partial_a} S)(\partial_b, \partial_c)$, so that
\[
	(\nabla_a S)_{bc}=
	\partial_a(S(\partial_b,\partial_c))-
	S(\nabla_a\partial_b,\partial_c)-
	S(\partial_b,\nabla_a\partial_c).
\]

The Cotton tensor has the following symmetries:
\begin{equation}\label{symmetries of Cotton}
 \begin{array}{rcl}
  C_{ijk} &=& -C_{jik} \\
  C_{ijk} + C_{jki} +C_{kij} &=& 0 \\
  g^{ij}C_{ijk} &=& 0\\
  g^{ik}C_{ijk} &=& 0.
\end{array}
\end{equation}

The Cotton tensor is conformally invariant only in dimension $3$, and indeed, in dimension $3$, a metric with vanishing Cotton tensor is conformally flat.


The Cotton tensor is equivalent to the so called \emph{Cotton-York tensor}.
This new tensor is defined by considering the Cotton tensor as a map $C_p:T_pM \to \Lambda^2(T_p^*M)$  (thanks to the anti-symmetry of $C$ with respect to its first two entries) and composing with the Hodge star operator $*:\Lambda^2(T_p^*M)\to T_p^*M$.
This gives a $(0,2)$ tensor that turns out to be symmetric and trace-free, but not conformally invariant, given by
\begin{equation}
\label{cotton-yorke}
CY_{ij}=\frac{1}{2}C_{kli} g_{jm} \frac{\epsilon^{klm}}{\sqrt{\det g}}=g_{jm}\left(\nabla_k S \right)_{li}\frac{\epsilon^{klm}}{\sqrt{\det{g}}}.
\end{equation}

It follows from (\ref{symmetries of Cotton}) that this tensor is symmetric and its trace is zero:
$$
CY_{ij}=CY_{ji}
$$
$$
g^{ij}CY_{ij}=CY^i_i=0.
$$

Lemma \ref{perturb metric to get any algebraic cotton} shows that for any sufficiently small symmetric traceless matrix, we can find a metric whose Cotton-York tensor at $p$ has that matrix in the canonical basis.
In particular, the Cotton-York tensor has no more pointwise symmetries, and we call the space $S_0^2(V)$ of symmetric traceless bilinear operators on an euclidean space $V$ the \emph{space of Cotton-York tensors} of $V$.


\section{Transversality theory}
\label{intro transversality}

In this work we want to study the transversality properties of the Weyl sections and Cotton-York sections, in order to prove Theorems \ref{size local} and \ref{size local dim=3}.
Let us recall the main results of transversality theory.
The reader can find more details in the very accessible book \cite{Hirsch}.

\subsection{Weak topology in the space of smooth functions.}

\begin{defn}\cite[section 2.1]{Hirsch}\label{defn: C r topologies}
 Let $M$ and $N$ be $C^r$ manifolds, and $C^r(M,N)$ be the set of $C^r$ maps from $M$ to $N$.

 Given $f\in C^r(M,N)$, charts $(\varphi, U)$ for $U\subset M$ and $(\psi, V)$ for $V\subset M$, a compact set $K\subset U$ such that $f(K)\subset V$, and $\varepsilon>0$, we define the set $\mathcal{N}^r(f;(\phi,U),(\psi,V),K,\varepsilon)$, consisting of all functions $g\in C^r(M,N)$ such that $g(K)\subset V$ and
 $$
 \| D^k (\psi \circ g \circ \varphi^{- 1}) (x) - D^k (\psi \circ f \circ \varphi^{- 1}) (x) \| \leqslant \varepsilon
 $$
 for all points $x\in K$, $k=0,\dots r$.

 The \emph{weak}, or \emph{compact-open $C^r$} topology, on $C^r(M,N)$, is the topology generated by the basis sets $\mathcal{N}^r(f;(\phi,U),(\psi,V),K,\varepsilon)$.
\end{defn}

\begin{defn}
 Let $P\rightarrow M$ be a $C^r$ smooth bundle, and $\Gamma^r(P\rightarrow M)$ be the set of $C^r$ sections of $P\rightarrow M$.

 The \emph{weak}, or \emph{compact-open $C^r$} topology, on $\Gamma^r(P\rightarrow M)$, is the topology induced by the inclusion of $\Gamma^r(P\rightarrow M)$ into $C^r(M,P)$, with the weak $C^r$ topology.
\end{defn}

\begin{rem}
 Both topologies agree for the trivial bundle $P=M\times N$, if we identify $\Gamma^r(M\times N\rightarrow M)$ with $C^r(M,N)$ in the usual way:
$$
\begin{array}{rcl}
 C^r(M,N) & \rightarrow & \Gamma^r(M\times N\rightarrow M) \\
 f & \rightarrow & u(p) = (p,f(p))
\end{array}
$$
$$
\begin{array}{rcl}
\Gamma^r(M\times N\rightarrow M) & \rightarrow & C^r(M,N) \\
u & \rightarrow & f(p) =  \pi_2\circ u(p)
\end{array}
$$
where $p$ is a point in $M$ and $\pi_2:M\times N\rightarrow N$ is the projection onto the second factor.
\end{rem}

\begin{rem}
 There is also a different, natural topology for $C^r(M,N)$ and $\Gamma^r(P\rightarrow M)$, called the strong topology, but it agrees with the weak topology when $M$ is compact, as we assume in this paper.
\end{rem}

\begin{defn}\cite[section 2.1]{Hirsch}
 If $M$ and $N$ are smooth manifolds, the $C^\infty(M,N)$ (or $C(M,N)$) topology is defined as the union of all the $C^k(M,N)$ topologies.

 The $\Gamma^\infty(P\rightarrow M)$ can be defined either as the subspace topology inherited from $C^\infty(M,N)$, or as the union of the topologies $\Gamma^r(P\rightarrow M)$.
\end{defn}

We will use the $C^\infty$ topologies for the rest of the paper.

\subsection{Transverse maps and sections.}


\begin{defn}\cite[pg 22]{Hirsch}
 Let $f:M\rightarrow N$ be a smooth map, $A\subset N$ a smooth submanifold and $K\subset M$ an arbitrary subset.
 We say $f$ is \emph{transverse} to $A$ at $x\in M$ if and only if either $f(x)\notin A$ or
 $$
 T_y A + d_x f(M_x) = T_y N.
 $$

 We say $f$ is \emph{transverse} to $A$ along $K$ (and write $f\PF_K A$) if and only if $f$ is transverse to $A$ at every $x\in K$.
 We write $f\PF A$ for $f\PF_M A$.

 We define $\PF_K (M, N ; A) $ as the set of all maps $f:M\rightarrow N$ transverse to $A$ along $K$, and $\PF (M, N ; A) $ as $\PF_M (M, N ; A) $.

 Let $Q$ be a sub-bundle of $P$ whose typical fiber is a smooth manifold.
 A section $u:M\rightarrow P$ is transverse to $Q$ at $x\in M$ if and only if $u$ is transverse to $Q$ at $x$ as a map from $M$ into $P$ to the total space of $Q$, which is a smooth manifold.

\end{defn}

\begin{rem}
 A section $u:M\rightarrow M\times N$ of a trivial bundle is transverse to the sub-bundle $M\times A$, for a smooth submanifold $A$ of $N$ if and only if the associated function $\pi_2\circ u:M \rightarrow N $ is transverse to $A$.
\end{rem}


\subsection{Stratified sets and stratified bundles}
So far, we have only defined transversality to a smooth submanifold $A\subset M$.
For most applications of transversality, this is enough, but for the results in this paper we will have to consider the less common notion of transversality to a smooth stratification.
The reason is that the set of Weyl operators with the eigenflag property is not a smooth manifold, but it has the structure of smooth stratification.

The required definitions and theorems for stratifications are indeed quite similar to those for smooth manifolds, and in this paper we will only use results about transversality to submanifolds, which can be found in the standard reference \cite{Hirsch}.


\begin{defn}
A smooth \emph{stratification} of a closed set $S\subset N$, for a manifold $N$, is collection of disjoint smooth submanifolds $S_j$ of dimension $j$ such that $S_0\cup \dots \cup S_k=S$ (some of the $S_j$ may be empty).
$S$ is called a \emph{stratified set}, and the $S_j$ are called \emph{strata}.
The \emph{dimension} of $S$ is the maximum dimension of a non-empty strata.

We further require that $S_0\cup \dots \cup S_j$ is a closed set for each $j=0\dots k$.
\end{defn}

We will not make use of the following property, but it is central in the theory of stratifications, and is mentioned in all the references (though sometimes with a different name):

\begin{defn}
 A smooth stratification  is \emph{regular} (or satisfies \emph{Whitney's A condition}) if and only if whenever $x_n\rightarrow y$, for $x_n\in S_j $ and $y\in S_{j-1}$, and the tangents to $S_j$ at $x_n$ converge to a space $\tau$, then $\tau$ contains the tangent to $S_{j-1}$ at $y$.
\end{defn}

\begin{rem}
 A smooth manifold with boundary is a regular stratification, with two strata consisting of the interior and the boundary.
 Any semialgebraic or semianalytic subset $S$ of $\R^n$ can be stratified, and the stratification is regular (see \cite[page 336]{Wall}).
\end{rem}


\begin{defn}
Let $\pi:P\rightarrow M$ be a smooth vector bundle with typical fiber $L$, and let $S\subset L$ be a subset invariant under the action of the structure group of the vector bundle.

Then the \emph{sub-bundle of $P\rightarrow M$ associated to $S$} is the subset $R$ of the total space consisting of the points that map to $S$ by any trivialization.
For any two trivializations $\psi:\pi^{-1}(U)\rightarrow U\times L$ and $\phi:\pi^{-1}(V)\rightarrow V\times L$, the induced map on the fiber over $p\in U\cap V$ is $\phi\circ \psi^{-1}|_{\{p\}\times L}$, which belongs to the structure group of $P\rightarrow M$.
Hence $R$ is well defined, and it is clear that the restriction of $\pi$ to $R$ defines a fiber bundle with typical fiber $S$.


A smooth \emph{stratified sub-bundle} is a sub-bundle associated to a subset $S$ that admits a stratification where each stratum is invariant under the action of the structure group.
\end{defn}

%
%

\begin{rem}
 A stratified sub-bundle is the union of the fiber sub-bundles $R_j\rightarrow M$ associated to the strata $S_j$ of $S$.
 Since the typical fiber of each $R_j$ is the smooth manifold $S_j$, the total space $R$ of the bundle is stratified by the total spaces $R_j$ of the sub-bundles.
\end{rem}

\begin{defn}
 A smooth map $f: M\rightarrow N$ is \emph{transverse} to a smooth stratification of a set $S\subset N$ if and only if it is transverse to each strata $S_j$.

 We define $\PF_K (M, N ; S)$ as the set of all maps transverse to $S$ along $K\subset M$.

 A smooth section $u:M \rightarrow P $ of $P\rightarrow M$ is \emph{transverse} to a stratified bundle $R\rightarrow M$ if and only if it is transverse to each of the sub-bundles $R_j\rightarrow M$ that stratify $R\rightarrow M$.

 We define $\PF_K(P;R) $ as the set of all sections of $P$ that are transverse to $R$ along $K\subset M$, and $\PF (P ; R) $ as $\PF_M(P ; R) $.
\end{defn}

The following lemma is straightforward:
\begin{lem}\label{lemma: equivalent definitions of transversality to a stratified bundle}
 Let $u:M \rightarrow P $ be a smooth section of a smooth vector bundle $\pi:P\rightarrow M$ with typical fiber $V$, and let $R\rightarrow M$ be a stratified sub-bundle of $P\rightarrow M$ associated to $S\subset V$.
 The following are equivalent:
 \begin{itemize}
  \item $u$ is \emph{transverse} to $R\rightarrow M$.
  \item For any trivialization of the bundle $\psi: \pi^{-1}(U)\rightarrow U\times V$, the map $\pi_2\circ\psi \circ u|_U:U\rightarrow V$ is transverse to each strata $S_j$ of $S$.
  \item $u$ is transverse to the total space $R$ of the bundle as a smooth map (recall that by definition, this means that $u$ is transverse to the total space of each $R_j$).
 \end{itemize}
\end{lem}

The importance of transversality is clear from the following lemma.
Only its last item is not standard, and this is all that we will need for stratifications.
The reader can find a similar result in exercises 3 and 15 of section 3.2 of \cite{Hirsch}.

\begin{lem}\label{preimage through a transverse map}
Let $M$ and $N$ be smooth manifolds, with $M$ compact, $A\subset N$ a smooth submanifold, $S\subset N$ a stratified set, and $f:M\rightarrow N$ a smooth map.
\begin{itemize}

 \item  If $f$ is transverse to $A\subset N$, then $f^{-1}(A)$ is either empty, or a smooth submanifold of $M$ with the same codimension of $A$.
 In particular, if $dim(M)<codim(A)$, then $f^{-1}(A)$ is empty.


 \item If $f$ is transverse to $S$, then $f^{-1}(S)$ is either empty, or a smooth stratification of $M$ with the same codimension of $S$.
 In particular, if $dim(M)<codim(S)$, then $f^{-1}(S)$ is empty.

 \item  Assume that only the highest dimensional stratum $S_k$ has a codimension smaller or equal than $dim(M)$.
 If $f$ is transverse to $S$, then $f^{-1}(S)$ is either empty, or a smooth compact submanifold of $M$ with codimension $codim(S_k)=dim(N)-k$.
\end{itemize}

\end{lem}
\begin{proof}
The first result can be found in \cite[Section 1, Theorem 3.3]{Hirsch}.

For the second point, we remark that each $f^{-1}(S_j)$ is a submanifold by the previous point, so $f^{-1}(S)=\cup_j f^{-1}(S_j)$ is partitioned into submanifolds whose codimension is $dim(N)-j$.
For any $j$, $f^{-1}(S_0)\cup\dots\cup f^{-1}(S_j) = f^{-1}(S_0\cup\dots\cup S_j)$, which is closed.

Assume now that $S=S_0\cup\ldots\cup S_j\cup S_k$ is a smooth stratification where the dimension of each stratum, except the top dimensional one, is less that $dim(N)-dim(M)$, and let $f:M\rightarrow N$ be a smooth map transverse to each stratum.
By the previous item, $f$ does not intersect the closed set $S_0\cup\dots\cup S_j$, hence $f^{-1}(S)=f^{-1}(S_k)$.
It follows that $f^{-1}(S_k)$ is a closed submanifold of $M$, so it is also compact.

\end{proof}

The following is a straightforward extension of the above lemma for stratified bundles:

\begin{lem}\label{preimage through a transverse section}
Let $R\rightarrow M$ be a stratified sub-bundle of $P\rightarrow M$ and $u:M\rightarrow P$ a smooth section.
Let $R\rightarrow M$ be a stratified sub-bundle associated to a set $S$ which admits a smooth stratification $S=S_0\cup\dots\cup S_j\cup S_k$, where $j<dim(P) - 2\,dim(M)$.

Then $u^{-1}(R)$ is a smooth compact submanifold of $M$ of codimension $dim(P)-dim(M)-k$.
\end{lem}
\begin{proof}
By Lemma~\ref{lemma: equivalent definitions of transversality to a stratified bundle}, $u:M\rightarrow P$ is transverse as a map to the total space of $R$, which is a manifold stratified by the total spaces of the sub-bundles $R_j$, whose dimension is $dim(M)+j$.
The result follows by the last item of the previous lemma.
\end{proof}

\subsection{The set of transverse sections}

We recall that a set $S\subset\R^n$ is \emph{residual} if it can be expressed as a countable intersection of open sets.
A diffeomorphism carries residual sets to residual sets, which makes it possible to translate this notion to open sets of manifolds.
Indeed, this notion makes sense in any topological space such that every intersection of a countable collection of open dense sets is dense.
These spaces are called \emph{Baire spaces}, and the space $C^\infty(M,N)$ is a Baire space (see \cite[section 2.4]{Hirsch}).

This is the main result of transversality theory:

\begin{thm}\label{main standard transversality theorem}
  Let $M$ and $N$ be smooth manifolds, and $A$ be a smooth submanifold of $N$.
  \begin{enumerate}

    \item $\PF (M, N ; A)$ is residual in $C^\infty(M,N)$.

    \item If $M$ is compact and $A$ is closed in $N$, then $\PF (M, N; A)$ is open and dense.

  \end{enumerate}
\end{thm}
\begin{proof}
 See \cite[3.2.1]{Hirsch}.
\end{proof}

It follows immediately from the first item that the set of maps transverse to a stratified set is residual.
However, the individual strata of a smooth stratification may not be closed, so we cannot use the second item in the previous lemma.
Exercise 15 in section 3.2 of \cite{Hirsch} asserts that $\PF (M, N; S)$ is open if the stratification is regular.
However, we will not use that hypothesis, so we prove directly that $\PF (M, N; S)$ is open under some hypothesis that hold in our situation:

\begin{thm}\label{set of transverse sections to stratifications is open and dense}
  Let $M$ and $N$ be smooth manifolds, with $M$ compact.
  Let $S_0\cup\ldots S_{i}\ldots\cup S_k$ be a smooth stratification of a closed set $S\subset N$ where $dim(S_i)=i$
  \begin{enumerate}
    \item Let $A$ be a smooth submanifold of $N$ such that $dim(M)<codim(A)$.
    Then $\PF (M, N; A)$ is open and dense.
    \item Assume that only the highest dimensional stratum $S_k$ has a codimension smaller or equal than $dim(M)$.
    Then $\PF (M, N; S)$ is open and dense.
  \end{enumerate}
\end{thm}
\begin{proof}

 We know from \ref{preimage through a transverse map} that for $f$ in $\PF (M, N; A)$, $f(M)$ is disjoint with $A$, so the distance between them is some $\varepsilon>0$.
 We can cover $M$ by finitely many compact sets $K_i$ such that each $K_i$ is contained in an chart $(U_i,\phi_i)$ and $f(K_i)$ is contained in a chart $(V_i,\psi_i)$.
 For any $r\geq 1$, the set $\bigcap_i\mathcal{N}^r(f;(\phi,U_i),(\psi,V_i),K_i,\varepsilon/2)$ is open in the $C^\infty(M,N)$ topology.
 For any $g$ in this set, $g(M)$ is disjoint with $A$.

 This proves that $\PF (M, N; A)$ is open, and it is dense by Theorem~\ref{main standard transversality theorem}.

 Assume now that $S=S_0\cup\dots\cup S_k$ is a closed smooth stratification where the dimension of each stratum, except the top dimensional one, is less that $dim(N)-dim(M)$, and let $f:M\rightarrow N$ be a smooth map transverse to each stratum.
 Then $f(M)$ is disjoint to the closed set $S_0\cup\dots\cup S_j$, hence their distance is some $\varepsilon>0$.

 Once again, cover $M$ by finitely many compact sets $K_i$ such that each $K_i$ is contained in a chart $(U_i,\phi_i)$ and $f(K_i)$ is contained in a chart $(V_i,\psi_i)$.

 For any element $g$ of $\mathcal{U}=\bigcap_i\mathcal{N}^r(f;(\phi,U_i),(\psi,V_i),K_i,\varepsilon/2)$, $g(M)$ is disjoint with $S_0\cap\dots\cap S_j$.
 It follows that $\mathcal{U}\subset \bigcap_{h=0}^{j}\PF (M, N; S_h)$.
 We can compute
 $$
 \begin{array}{rcl}
 \mathcal{U}\cap \PF (M, N; S) &=& \mathcal{U}\cap \PF (M, N; S_k) \\
 &=& \mathcal{U}\cap \bigcap_{i}\PF_{K_i} (M, N; S_k) \\
 &=& \bigcap_i\left(\mathcal{N}^r(f;(\phi,U_i),(\psi,V_i),K_i,\varepsilon/2)\cap \PF_{K_i} (M, N; S_k)\right).
 \end{array}
 $$
 But for each $i$
 $$
 \begin{array}{ll}
 \mathcal{N}^r(f;(\phi,U_i),(\psi,V_i),K_i,\varepsilon/2)\cap \PF_{K_i} (M, N; S_k) &=\\
 \mathcal{N}^r(f;(\phi,U_i),(\psi,V_i),K_i,\varepsilon/2)\cap (\iota_{U_i}^{\ast})^{-1}\left(\PF_{K_i} (U_i, N; S_k)  \right) &=\\
 \mathcal{N}^r(f;(\phi,U_i),(\psi,V_i),K_i,\varepsilon/2)\cap (\iota_{U_i}^{\ast})^{-1}\left(\PF_{K_i} (U_i, V_i; S_k) \right).
 \end{array}
 $$
 where $\iota_{U_i}:U_i\rightarrow M$ is the inclusion and $\iota_{U_i}^\ast:C(M,N)\rightarrow C(U_i, N) $ sends $f$ to $f\circ\iota_{U_i}$, so that $(\iota_{U_i}^\ast)^{-1}\left(\PF (U_i, N; S_k)\right)  $ is the set of maps whose restriction to $U_i$ is transverse to $S_k$ along $K_i$.

 The map $\iota_{U_i}$ is continuous and $\PF_{K_i} (U_i, V_i; S_k)$ is open by lemma \cite[3.2.3]{Hirsch}.
 Thus $\mathcal{U}\cap \PF (M, N; S)$ is an open neighborhood of $f$.
 It follows that $\PF (M, N; S)$ is open, and it is dense since by Theorem~\ref{main standard transversality theorem}, each $\PF (M, N; S_i)$ is residual and $C(M,N)$ is a Baire space.
%
\end{proof}

Finally, we also need the so called \emph{parametric transversality} results:

\begin{thm}\label{parametric transversality}
  Let $B$ be a smooth manifold, $S\subset N$ a stratified set.
  Let $F : B \times M \rightarrow N$ be a smooth map transverse to $S$.

  Define the functions $F_b : M \rightarrow N$ by $F_b (x) = F (b, x)$.

  Then the set
  \[ \PF (F ; S) = \{ b \in B : F_b \PF S \} \]
  is residual.
\end{thm}
\begin{proof}
 By definition, $\PF (F ; S)=\bigcap_j \PF (F ; S_j)$, and \cite[3.2.7]{Hirsch} shows that each set $\PF (F ; S_j)$ is residual in $V$.

\end{proof}

\section{Proof of Theorem \ref{size local}}

It may be tempting to think that a \emph{generic metric} has at every point a Weyl tensor without the eigenflag property.
This is actually true in dimension $5$ and above, but not in dimension $4$, as we shall see later.
However, if a metric admits a local LCW at one point $p$, then the Weyl tensor will have the eigenflag property at all points near $p$.
Thus, we only need to prove that for a generic metric, the set of points whose Weyl tensor does not have the eigenflag property is dense.
Both in dimension $4$ and higher, the result follows by the transversality arguments in the previous section.

Let $V$ be an euclidean space, $\Lambda^2 V$ the associated space of bivectors, $S^2(\Lambda^2 V)$ the symmetric operators on the space of bivectors, and $\mathcal{W}(V)$ be the intersection of the kernels of the Bianchi map $b_V$ (\ref{eq:bianchi}) and the Ricci map $r_V$ (\ref{eq:ricci}) on $S^2(\Lambda^2 V)$.

We recall the following purely algebraic statement from \cite{AFGR}:

\begin{thm}[{\cite[Theorem 6.1]{AFGR}}]\label{eigenflag property is semialgebraic}
 The subset $\mathcal{EW}(V)$ of the Weyl tensors on $V$ that has the eigenflag property is a semialgebraic subset of the space of Weyl tensors.

 Its codimension is exactly:
 $$\frac{1}{3} \, n^{3} - n^{2} - \frac{4}{3} \, n + 2.$$
 In particular, the codimension is $2$ for $n=4$ and $12$ for $n=5$.
 It is greater than $n$ for $n\geq 5$.
\end{thm}

Semialgebraic sets are defined by any combination of polynomial equations and inequalities.
They also appear as projections of algebraic sets.
A projection of a real algebraic set need not be an algebraic set, but it always is semialgebraic, by the Tarski-Seiderberg theorem.
As an example, if we project the circle $\{(x,y)\in\R^2:x^2+y^2=1\}$ onto the $x$ axis, we get the closed interval $[-1,1]$, which is semialgebraic but not algebraic.

%

Since the space of Weyl operators depends on the point, we have to work with the vector bundle of Weyl curvature operators, and the fiber bundle of Weyl operators with the eigenflag property.



\begin{defn}\label{defn: weyl section etc}
 Given a Riemannian manifold $(M,g)$, the \emph{curvature bundle} $\mathcal{R}\rightarrow M $ is the vector bundle whose fiber at $p$ is the kernel $\mathcal{R}_p$ of the Bianchi map $b_p:S^2(\Lambda^2(T_p M))\rightarrow \Lambda^4(T_p M) $.

 The \emph{Weyl bundle} $\mathcal{W}\rightarrow M $ is the vector sub-bundle of the curvature bundle whose fiber at $p$ is the kernel $\mathcal{W}_p$ of the Ricci contraction $r_p: \mathcal{R}_p \rightarrow S^2(T_p M)$.
\end{defn}

\begin{defn}\label{defn: eigenflag bundle etc}
 The \emph{eigenflag bundle} $\mathcal{EW}\rightarrow M $ is the sub-bundle of $\mathcal{W}\rightarrow M $ associated to the subset of $\mathcal{W}(\R^n)$ where the Weyl operator has the eigenflag property.

 The \emph{Weyl section} of a metric $g$ is the section of the Weyl bundle that maps the point $p$ to its Weyl operator $W_p$ at $p$.
\end{defn}

The curvature bundle is defined for a smooth manifold, and does not depend on the choice of a Riemannian metric on the manifold.
The Weyl bundle, however, does depend on $g$.
This is an inconvenient property for the purposes of this paper: if we define the map that sends a metric $g$ to its Weyl section $p\rightarrow W_p$, we would have to use as target space the set $\Gamma(\mathcal{R})$ of sections of the full curvature bundle.
Then the property of a metric being eigenflag-transverse would not be identified with transversality of that map to the set of sections of a fixed sub-bundle of $\mathcal{R}$, and we could not apply easily the standard results in transversality theory, like theorems \ref{main standard transversality theorem} and \ref{set of transverse sections to stratifications is open and dense}.

In order to overcome this technical difficulty, we define the \emph{Weyl map} from the space of Riemannian metrics into the space of sections of a fixed \emph{extended} vector bundle, and study when the section that corresponds to a metric is transverse to a fixed stratified sub-bundle.


\begin{defn}\label{defn: extended Weyl bundle}
  A Riemannian metric is a section of the vector bundle whose fiber at $p$ is the set $S^2(T_p M)$ of symmetric operators on $T_pM$.
  Furthermore, Riemannian metrics must be positive definitive, and this amounts to restricting to an open set (in the $C^0$ topology) of sections of the bundle $S^2(T M)$.
  This is the \emph{space of Riemannian metrics} $\mathcal{G}(M)$, and we always consider it with the topology inherited from the compact open $C^\infty$ topology in $\Gamma^\infty(S^2(T M))$.

  The \emph{extended Weyl bundle} is the vector bundle whose fiber at $p$ is:
$$
\widetilde{\mathcal{W}_p}= \{
  (g_p,W_p)\in S^2(T_p M)\times \mathcal{R}_p:
  r_{g_p}(W_p)=0
\},
$$
  where $r_{g_p}$ is the Ricci map for $T_p M$ and the metric $g_p$.

  The \emph{extended Weyl section} of a metric $g$ is the section of the extended Weyl bundle that sends $p$ to the pair $(g_p, W_p)$.

  The \emph{Weyl map} sends a metric $g$ to its extended Weyl section, and is a continuous map from the space of Riemannian metrics into the space of smooth sections of the extended Weyl bundle

$$\mathbb{W}:\mathcal{G}(M) \rightarrow \Gamma^\infty(\widetilde{\mathcal{W}}).$$
\end{defn}

\begin{defn}\label{defn: extended eigenflag bundle}
 The \emph{extended eigenflag bundle} $\widetilde{\mathcal{EW}}\rightarrow M $ is the sub-bundle of $\widetilde{\mathcal{W}}\rightarrow M $ associated to the subset of $S^2(T_p M)\times\mathcal{W}_p$ where the Weyl operator has the eigenflag property.
 We will see in Lemma~\ref{eigenflag bundle is stratified} that $\widetilde{\mathcal{EW}}\rightarrow M $ is a stratified bundle.

 A metric is \emph{eigenflag-transverse} if its extended Weyl section is transverse to the extended eigenflag bundle.
\end{defn}

\begin{rem}\label{rem: we can work with the extended eigenflag bundle}
 For a fixed metric $g$, the Weyl bundle for $g$ can be identified with a subset of the extended Weyl bundle in a natural way.
 Its intersection with the extended eigenflag bundle is the eigenflag bundle for $g$.
 The tangent to the Weyl bundle for $g$ contains the tangent to the factor $\mathcal{W}_p$ of the fiber to the extended Weyl bundle.
 The tangent to the extended eigenflag bundle contains the tangent to the other factor $S^2(T_p M)$.
 Hence the Weyl bundle for $g$ is transverse to the extended eigenflag bundle.

 Thus, the eigenflag bundle for $g$ inherits a stratification from a stratification of the extended eigenflag bundle, consisting of the intersections of the strata of the extended eigenflag bundle with the eigenflag bundle for $g$.

 The Weyl section of a metric $g$ is transverse to the eigenflag bundle for $g$ if and only if the metric is eigenflag-transverse.
\end{rem}



\begin{lem}\label{eigenflag bundle is stratified}
 The extended eigenflag bundle is a stratified sub-bundle of the extended Weyl bundle.
\end{lem}
\begin{proof}
 The general linear group $GL(T_p M)$ acts on $S^2(T_p M)\times \mathcal{R}(T_p M)$ and preserves the fiber $\widetilde{\mathcal{W}_p}$ of the extended Weyl bundle:
$$
\begin{array}{rcl}
 \rho_L(R)(x\wedge y, z\wedge t) &=& R(L(x)\wedge L(y), L(z)\wedge L(t)) \\
 \rho_L(g)(x, y) &=& g(L(x), L(y)).
\end{array}
$$
 The set of pairs $(g,W)$ where $W$ has the eigenflag property is clearly invariant under this action.

 By Theorem~\ref{eigenflag property is semialgebraic} and lemma \ref{a semialgebraic invariant set admits an invariant stratification}, the set of Weyl operators with the eigenflag property admits a stratification that is invariant under the action of the structure group of the tangent bundle.
 The lemma follows from the definition of stratified bundle.
\end{proof}

There are many proofs in the literature (see \cite{Wall} or \cite[9.2.1]{BochnakCosteRoy} for instance) that a semialgebraic set such as $\mathcal{EW}(V)$ admits a smooth stratification (and in fact, the stratification is regular).
However, this stratification may not be invariant under the action of $GL(V)$.

\begin{lem}\label{a semialgebraic invariant set admits an invariant stratification}
 Let $S\subset V$ be a closed semialgebraic subset of a real vector space.
 Let $G$ be a group and let $\rho:G\times V\rightarrow V$ be an action of $G$ on $V$ by smooth maps such that for every $g\in G$, we have $g(S)=S$.

 Then $S$ admits a smooth stratification $S=S_1\cup\ldots\cup S_k$ such that $g(S_j)=S_j$ for every $j=1\dots k$ and $g\in G$.
\end{lem}

\begin{proof}
The proof is inspired in the simple proof of a similar result in pages 336 and 337 of \cite{Wall}.

We proceed by induction on $d=\dim(S)$.
We say a point $p\in S$ is regular if $S\cap U$ is a $C^1$ manifold for some neighborhood $U$ of $p$.
We split $S$ into the set of regular points $S_{reg}$ and its complement $S_{sing}$.
Since the action $\rho$ is by globally invertible diffeomorphisms, $S_{sing}$ and $S_{reg}$ are invariant under the action.
It is clear that $S_{sing}$ is closed.

We can use theorem 2.1 in \cite{KoikeShiota} to show that $S_{sing}$ is a semialgebraic subset of $S$, if we consider any semialgebraic set $N$ (e.g., a point), and let $f:S\rightarrow N$ be a constant map.
Then $f^{-1}(f(x))=S$, and $\Sigma_1 = S_{sing}$.

Since it is a semialgebraic set, $S_{sing}$ is a finite disjoint union of Nash manifolds $N_\alpha$.
The dimension of each $N_\alpha$ can be at most $d-1$, since otherwise any point of a strata $N_\alpha$ with dimension $d$ that is not in the closure of the other strata would be regular (see \cite[2.4]{KoikeShiota}).


By induction on the dimension of the semialgebraic set, $S_{sing}$ admits a smooth stratification $S_{sing}=S_0\cup\ldots\cup S_{d-1}$ where all the strata are invariant under the action $\rho$.

Then $S$ can be stratified by the disjoint subsets $S_0,\ldots, S_{d-1},S_{reg}$, which are smooth and invariant under the action.
For $j=0\dots d$, the union $S_1\cup\ldots\cup S_{j}$ is closed by the induction hypothesis.
\end{proof}

We also need the following lemma from \cite{AFGR}:

\begin{lem}[{\cite[Lemma 6.5]{AFGR}}]\label{perturb metric to get any algebraic curvature}
  Let $M$ be a Riemannian manifold with metric $g$ and $p$ any point in $M$, with $R_p$ the curvature of the metric $g$ at $p$.
  Let $\varphi$ be a cutoff function with support contained in a neighborhood of $p$ which admits normal coordinates $x^h$.

  Then there is a number $\varepsilon>0$ such that, for any algebraic curvature operator $R^*$ whose norm as a $(4,0)$-tensor is smaller than $\varepsilon$, the following defines a Riemannian metric
\[
g'_{ij}=g_{ij}-\frac{1}{4}\sum_{k,h} R^*_{ihjk} x^h x^k \varphi(x),
\]
  whose curvature at $p$ is $R_p + R^*$.

 \end{lem}



 \begin{proof}[Proof of Theorem \ref{size local}]
 We first remark that it is enough to give the proof for $U=M$.

 Let $\ETM$ be the subset of $\mathcal{G}(M)$ consisting of eigenflag transverse metrics, and let $\PF(\widetilde{\mathcal{W}}, \widetilde{\mathcal{EW}}) $ be the set of sections of the extended Weyl bundle that are transverse to the extended eigenflag bundle.

 By the transversality Theorem~\ref{main standard transversality theorem} 
 and Lemma~\ref{eigenflag bundle is stratified}, $\PF(\widetilde{\mathcal{W}}, \widetilde{\mathcal{EW}}) $ is residual in $\Gamma^\infty(\widetilde{\mathcal{W}})$.
 However, we remark that $\ETM\neq \PF(\widetilde{\mathcal{W}}, \widetilde{\mathcal{EW}})$, because an arbitrary section of the extended Weyl bundle may not be the Weyl section of any metric.
 As an example, let $g_E$ be the euclidean metric in $\R^n$, and choose any nonzero bilinear operator $W_0\in\mathcal{W}(\R^n)$.
 Then $x\rightarrow (g_E, W_0) $ is not the Weyl section of any metric, since $W_0$ should be the Weyl tensor of the euclidean metric, which is zero.

 \strong{Proof that $\PF(\widetilde{\mathcal{W}}, \widetilde{\mathcal{EW}})$ is open in $\Gamma^\infty(\widetilde{\mathcal{W}})$}.

 If the dimension is $5$ or greater, Theorem~\ref{eigenflag property is semialgebraic} implies that the codimension of $\widetilde{\mathcal{EW}}$ in $\widetilde{\mathcal{W}}$ is greater than the dimension of $n$.
 The total space of $\widetilde{\mathcal{EW}}$ is a finite union of submanifolds $\{A_j\}_{j\in J}$ of the total space of $\widetilde{\mathcal{W}}$.
 For fixed $j$, the set of maps from $M$ into the total space of $\widetilde{\mathcal{W}}$ which are transverse to $\{A_j\}_{j\in J}$ is open by the first part of Theorem~\ref{set of transverse sections to stratifications is open and dense}, since $codim(A_j)>dim(M)$.
 Since $J$ is finite, the set $\PF(\widetilde{\mathcal{W}}, \widetilde{\mathcal{EW}}) $ is the intersection of an open set with the set of maps that are sections of the bundle, hence $\PF(\widetilde{\mathcal{W}}, \widetilde{\mathcal{EW}}) $ is open in $\Gamma(\mathcal{W})$.


 For dimension $4$, we start with the results in section 6.1 of \cite{AFGR} (see \cite[2.1]{AFG} for a similar result).
 The set of Weyl operators with the eigenflag property has the following decomposition:
 \begin{itemize}
  \item A nonsingular stratum of codimension $2$ in $\mathcal{W}$ where the operators diagonalize in a basis of simple bivectors, and the Weyl operator has three different eigenvalues, each of multiplicity $2$.
  \item A stratum of operators with two different eigenvalues: $\lambda$ of multiplicity $2$ and $-\lambda/2$ of multiplicity $4$.
  \item The zero operator.
 \end{itemize}
 The second stratum is parameterized by $\lambda$ and a $2$-plane.
 Since the dimension of the Grassmannian of $2$-planes is $2\cdot(4-2)=4$, we deduce that the set of Weyl operators with the eigenflag property has a top dimensional stratum of codimension $2$, a stratum of codimension $5$ (since the space of Weyl operators in dimension $4$ has dimension $10$) and a point.

 Once again, the total space of $\widetilde{\mathcal{EW}}$ is a finite union of submanifolds $\{A_j\}_{j\in J}$ of the total space of $\widetilde{\mathcal{W}}$, but this time exactly one of the strata $A_j$ has codimension smaller than $dim(M)=4$.
 Then by the second part of Theorem~\ref{set of transverse sections to stratifications is open and dense} and the same argument we used for dimension greater than 4, we see that the set of sections of $\widetilde{\mathcal{W}}$ that are transverse to $\widetilde{\mathcal{EW}}$ is open.

 \strong{Proof that $\ETM$ is open}.

 Let $g\in\ETM$.
 Its extended Weyl section has a neighborhood $\mathcal{U}$ in $\Gamma^\infty(\widetilde{\mathcal{W}})$ consisting only of eigenflag-transverse sections.
 Continuity of the Weyl map ensures that there is an open neighborhood of $g$ in $\mathcal{G}(M)$ such that any metric $g'$ in it has an extended Weyl section in $\mathcal{U}$, that is thus eigenflag-transverse.

 \strong{Proof that $\ETM$ is dense}.

 We will use the parametric transversality Theorem~\ref{parametric transversality}.
 Fix a Riemannian metric $g_0\in \mathcal{G}(M)$.
 We will build a finite dimensional space of Riemannian metrics on the manifold $M$ containing $g_0$ and parameterized by a map $G$ defined on a smooth manifold $B$ and continuous for the compact open $C^{\infty}$ topology.
 Since we intend to use Theorem~\ref{parametric transversality}, the map $G$ must be such that $GW(s,p) = \mathbb{W}(G(s))_p$ is transverse to $\widetilde{\mathcal{EW}}$.
 Indeed, the map $GW$ will be a submersion onto $\widetilde{\mathcal{W}}$.

 For any $p\in M$, let $\varphi_p$ be a smooth function such that $\varphi_p(p)=1$ and such that $U_p=supp(\varphi_p)$ admits normal coordinates $x^k$.

 We consider the map $G_p$ defined from $S^2(T_p M)\times\mathcal{R}(T_p M)$ into the space of symmetric bilinear operators on $M$, where $G_p (h,R)$ agrees with $g_0$ outside $U_p$ and is defined in $U_p$ by:
 $$G_p (h,R)_{ij} = (
  g_0)_{ij} +
  h_{ij} \varphi_p(x)
  - \frac{1}{4} \sum R_{ihjk} x^h x^k \varphi_p ( x).
 $$
 Clearly, $G_p (h,R)$ is a Riemannian metric if both $h$ and $R$ are small enough (in the sense that all the numbers $|h_{ij}|$ and $|R_{ijkl}|$ are smaller than some $\varepsilon>0$, for example).
 By lemma \ref{perturb metric to get any algebraic curvature}, the map that assigns to every $R\in\mathcal{R}(T_p M)$ the curvature of $G_p(0,R)$ at $p$ has a surjective differential.
 Its composition with the projection onto the Weyl part of the curvature is also surjective.
 It follows that the differential of $(h,R)\rightarrow (G_p(h,R)_p, W(G_p(h,R))_p)$, is surjective at $(0,0)$.

 Since this is an open condition, there is an open set $O_p\subset U_p$ such that the differential of $(h,R)\rightarrow \mathbb{W}(G_p(h,R))_q = (G_p(h,R)_q, W(G_p(h,R))_q)$ is surjective for any $q\in O_p$.

 We remark that if $M$ admits global coordinates, we can choose $\varphi(x)=1$ for all $x\in M$, but the differential of $(h,R)\rightarrow \mathbb{W}(G_p(h,R))_q$ may still not be surjective for all $q\in M$.
 However, since $M$ is compact, there is \emph{a finite family of points} $\{p_\lambda\}$, for $\lambda = 1 \ldots L$, such that $M \subset \cup O_{p_\lambda}$.
 The differential of $(h,R)\rightarrow \mathbb{W}(G_{p_\lambda}(h,R))_q$ at $(0,0)$ is surjective for any $q\in O_{p_\lambda}$.


 We can now build the set $B$ and the map $G$.
 The map $G$ is defined from $\Pi_{l=1}^{L}(S^2(T_{p_l}M)\times \mathcal{R}(T_{p_l}M))$ into the space of symmetric bilinear operators on $M$:
$$G ( h_1, R_1,\dots,h_L, R_L) =
  g_0 + \sum_{\lambda=1}^{L} (G_{p_\lambda}(h_\lambda,R_\lambda) - g_0).
$$
The tensor $G(s)$ is a Riemannian metric if $s = ( h_1, R_1,\dots,h_L, R_L)$ is sufficiently small.
Let $\tilde{B} \subset \Pi_{l=1}^{L}(S^2(T_{p_l}M)\times \mathcal{R}(T_{p_l}M))$ be a neighborhood of $\{ (0,0) \}^L$ such that $G(s)$ is a Riemannian metric for any $s\in B$.
$G$ induces the map:
$$
\begin{array}{rcl}
 \tmop{GW} : \tilde{B} \times M &\rightarrow& \widetilde{\mathcal{W}} \\
 (s,p) &\rightarrow& \mathbb{W}(G(s))_p = (G(s)_p, W(G(s))_p).
\end{array}
$$


If, for example, $p \in O_{p_1}$, we restrict $\tmop{GW}$ to $\tilde{B}\cap S^2(T_{p_1} M)\times\mathcal{R}(T_{p_1} M)\times \{ (0,0) \}^{L - 1} \times \{p\}$ and we recover the map $(h,R,0,\dots,0,p)\rightarrow \mathbb{W}(G_{p_1}(h,R))_p$.
We know from our construction that this map has a surjective differential at $h=0,R=0$ and it follows that the differential of $\tmop{GW}$ at $\{ (0,0) \}^L \times \{p\}$ is surjective for any $p$.
In particular, $\tmop{GW}$ is transverse to the extended eigenflag bundle if we restrict it to a sufficiently small neighborhood $B\subset\tilde{B}$ of $\{ (0,0) \}^L$.

 Thus the parametric transversality theorem \ref{parametric transversality} shows that we can find a parameter $s$ as small as we need, so that the metric $G(s)$ is eigenflag-transverse, and as close to $g_0$ as we want.
 The second part of Theorem~\ref{size local} follows.

 \strong{Proof that no metric in $\ETM$ admits a local LCW}.
 It follows from the above and Lemma~\ref{preimage through a transverse map} that in dimension $5$ and above, the Weyl tensor of an eigenflag-transverse metric never has the eigenflag property.
 In dimension $4$, the Weyl tensor of an eigenflag-transverse metric may have the eigenflag property at the points of a $2$ dimensional compact manifold.
 In both situations, the subset of $M$ consisting of points whose Weyl tensor does not have the eigenflag property is dense, and thus there cannot be any LCW defined in an open set.
 The first part of Theorem~\ref{size local} follows.
\end{proof}

\section{Proof of Theorem \ref{size local dim=3}.}

In dimension $3$, the Weyl tensor vanishes at every point, so we use the Cotton-York tensor instead.
We saw in section \ref{tensors} that the Cotton-York operator is symmetric and traceless.
The symmetric operators on $TM$ are defined independently of the metric, but the definition of the trace requires use of the metric.
We start with definitions analogous to the ones in \ref{defn: weyl section etc}:

\begin{defn}\label{defn: cotton york section etc}
 Given a Riemannian manifold $(M,g)$, the \emph{Cotton-York bundle} $\mathcal{CY}\rightarrow M $ is the vector bundle whose fiber at $p$ is the set of symmetric traceless operators $S^2_0(T_p M)$.

 The \emph{singular Cotton-York bundle} is the sub-fiber bundle $\mathcal{SCY}\rightarrow M $ of $\mathcal{CY}\rightarrow M $ whose fiber at $p$ is the subset of the operators in $S^2_0(T_p M)$ with zero determinant.

 The \emph{Cotton-York section} of a metric $g$ maps a point $p$ to its Cotton-York tensor $CY_p$.
\end{defn}

For the same reasons as in the previous section, we define extended bundles whose definitions do not require a particular metric.

\begin{defn}\label{defn: extended cotton york section etc}
 Given a manifold $M$, the \emph{extended Cotton-York bundle} $\widetilde{\mathcal{CY}}\rightarrow M $ is the vector bundle whose fiber at $p$ is
$$
\{
  (g_p, Y_p)\in S^2(T_p M)\times S^2(T_p M): \sum g^{ij}Y_{ij}=0
\}.
$$

 The \emph{extended singular Cotton-York bundle} is the sub-fiber bundle $\widetilde{\mathcal{SCY}}\rightarrow M $ of $\widetilde{\mathcal{CY}}\rightarrow M $ whose fiber at $p$ consists of those pairs $(g, Y)$ where $det(Y)$ is zero.

 We will see in \ref{singular cotton-york bundle is stratified} that it is a stratified bundle.

 The \emph{extended Cotton-York section} of a metric $g$ is the section of the extended Cotton-York bundle that assigns to each point $p$ the pair $(g_p, CY_p)$, where $CY_p$ is the Cotton-York tensor of $g$ at $p$.

 A metric is \emph{SCY-transverse} if its Cotton-York section is transverse to the singular Cotton-York bundle.
\end{defn}

For the same reasons mentioned in Remark~\ref{rem: we can work with the extended eigenflag bundle}, we only need to define transversality in the extended bundles.

\begin{defn}
  The \emph{extended Cotton-York map} sends a metric $g$ to its extended Cotton-York section, and is defined from the space of smooth Riemannian metrics into the space of smooth sections of the extended Cotton-York bundle:

$$\mathbb{CY}:\mathcal{G}(M) \rightarrow \Gamma^\infty(\mathcal{CY}).
$$
\end{defn}

\begin{lem}\label{singular cotton-york bundle is stratified}
 The extended singular Cotton-York bundle is a stratified sub-bundle of the extended Cotton-York bundle.
\end{lem}
\begin{proof}
 The proof can be done using the exact same ideas used in \ref{eigenflag bundle is stratified}, since both the set of Cotton-York tensors and the subset of singular Cotton-York tensors are invariant under the action of the general linear group:
$$
\begin{array}{rcl}
 \rho_L(g)(x, y) &=& g(L(x), L(y)) \\
 \rho_L(Y)(x, y) &=& Y(L(x), L(y)).
\end{array}
$$
\end{proof}

The following Lemma plays the role of Lemma \ref{perturb metric to get any algebraic curvature} for the Cotton tensor:

\begin{lem}[{\cite[Theorem 6.7]{AFGR}}]\label{perturb metric to get any algebraic cotton}
  Let $M$ be a Riemannian manifold of dimension $3$ with metric $g$ and $p$ any point in $M$.
  Let $\varphi$ be a cutoff function with support contained in a neighborhood of $p$ which admits normal coordinates $x^k$.

  Define the vector space of tuples of numbers $a_{ij}^{klm}$ invariant under permutations of the lower, and of the upper indices:
  $$
  \mathbb{A}=\{(A_{ij}^{klm}): i,j,k,l,m\in\{1,2,3\}, A_{ij}^{klm}=A_{\pi(ij)}^{\sigma(klm)}, \pi\in S_2,\sigma\in S_3\}.
  $$

  Then for any algebraic Cotton-York tensor $CY^0$ close enough to $CY_p$, there is $(A_{ij}^{jkl})\in\mathbb{A}$ such that for the metric $g'$:
   $$g_{ij}' = g_{ij} + \varphi\sum A_{ij}^{klm} x^k x^l x^m$$
  the Cotton-York tensor at $p$ is $CY^0$.
%
 \end{lem}

 \begin{proof}[Proof of Theorem \ref{size local dim=3}]
 The argument is now similar to the one for $\dim(M)\geq 4$.

 Let $\SCYTM$ be the subset of $\mathcal{G}(M)$ consisting of SCY-transverse metrics, and let $\PF(\widetilde{\mathcal{CY}}, \widetilde{\mathcal{SCY}})$ be the set of sections of the extended Cotton-York bundle that are transverse to the singular Cotton-York bundle.


\strong{Proof that $\PF(\widetilde{\mathcal{CY}}, \widetilde{\mathcal{SCY}})$ is open in $\Gamma^\infty(\widetilde{\mathcal{CY}})$}.

 We have to stratify the space of singular traceless symmetric operators (see \cite[3.2]{AFG} for a similar result).
 This time there are only two strata:
 \begin{itemize}
  \item A nonsingular stratum where the operator $Y$ has eigenvalues $\lambda, -\lambda, 0$.
  \item The zero operator.
 \end{itemize}
 For an operator in the first stratum, since the eigenvalues are different, the operator diagonalizes in a unique orthonormal basis (up to reordering the elements).
 We can cover the stratum with a mapping
 $$
  (\lambda,Q)\rightarrow Q\cdot \left(
  \begin{array}{ccc}
   \lambda & 0 & 0 \\
   0 & -\lambda & 0 \\
   0 & 0 & 0
   \end{array}
  \right)\cdot Q^{t}
 $$
 defined on $\R\times SO(3)$.
 Thus, the set of singular Cotton-York operators has a top dimensional strata of dimension $4$, and hence codimension $1$, and a point (of codimension $5$).
 Then the second part of Theorem~\ref{set of transverse sections to stratifications is open and dense} proves that the set of sections of $\widetilde{\mathcal{CY}}$ that are transverse to $\widetilde{\mathcal{SCY}}$ is open.

 \strong{Proof that $\SCYTM$ is open}.

 Let $g\in\SCYTM$.
 Its extended Cotton-York section has a neighborhood $\mathcal{U}$ in $\Gamma^\infty(\widetilde{\mathcal{CY}})$ consisting only of SCY-transverse sections.

 Continuity of $\mathbb{CY}$ ensures that there is an open neighborhood of $g$ in $\mathcal{G}(M)$ such that any metric $g'$ in it has an extended CY section in $\mathcal{U}$, that is thus SCY-transverse.
 In other words, the set of $SCY$-transverse metrics is open in $\mathcal{G}(M)$.

 \strong{Proof that $\SCYTM$ is dense}.

 Let $g_0$ be an arbitrary Riemannian metric on $M$.

 For any $p\in M$, let $\varphi_p$ be a smooth function such that $\varphi_p(p)=1$, such that $U_p=supp(\varphi_p)$ admits normal coordinates $x^k$.

 For every $h\in S^2(T_{p_l}M)$ and $A\in\mathbb{A}$, we define a metric $G_p (h,A)$ that agrees with $g_0$ outside $U_p$ and is defined in $U_p$ by:
$$
 (h_{ij}, A_{ij}^{klm}) \rightarrow
 G_p (h,A)_{ij} = ( g_0)_{ij} + \varphi_p(x)h_{ij} + \sum \varphi_p(x)A_{ij}^{klm} x^k x^l x^m
$$

 By Lemma~\ref{perturb metric to get any algebraic cotton}, there is an open neighborhood $O_p\subset U_p$ of $p$ such that
$$
(h, A) \rightarrow \mathbb{CY}(G_{p}(h,A))_q
$$
 has a surjective differential for any $q\in O_p$.

 We collect a finite family of points $\{p_\lambda\}$, for $\lambda = 1 \ldots L$, such that $M \subset \cup O_{p_\lambda}$.
 We define $\tilde{B}$ to be a neighborhood of $\{(0,0)\}^L$ in $\Pi_{l=1}^{L}(S^2(T_{p_l}M)\times \mathbb{A})$ such that
$$G ( h_1, A_1,\dots,h_L, A_L) =
  g_0 + \sum_{\lambda} (G_{U_\lambda}(h_\lambda,A_\lambda) - g_0)
$$
 is a Riemannian metric for any $( h_1, A_1,\dots,h_L, A_L)\in\tilde{B}$.
%
 $G$ induces a map:

$$
\begin{array}{rcl}
 \tmop{GCY} : \tilde{B} \times M &\rightarrow& \widetilde{\mathcal{CY}} \\
 (s,p) &\rightarrow& \mathbb{CY}(G(s))_p = (G(s)_p, CY(G(s))_p).
\end{array}
$$

If, for example, $p \in O_{p_1}$, we restrict $\tmop{GCY}$ to $S^2(T_{p_1}M)\times \mathbb{A} \times\{ (0,0) \}^{L - 1} \times \{p\}$, and recover the map $(h,A,0,\dots,0,p)\rightarrow \mathbb{CY}(G_{p_1}(h,A))_p$.
We know from our construction that this map has a surjective differential at $h=0,A=0$ and it follows that the differential of $\tmop{GCY}$ at $\{ (0,0) \}^L \times \{p\}$ is surjective for any $p$.
In particular, $\tmop{GCY}$ is transverse to the extended singular Cotton-York bundle if we restrict it to a sufficiently small neighborhood $B\subset\tilde{B}$ of $\{ (0,0) \}^L$.

Thus the parametric transversality theorem \ref{parametric transversality} shows that we can find a parameter $s$ as small as we need, so that the metric $G(s)$ is SCY-transverse, and as close to $g_0$ as we want.
The second part of Theorem~\ref{size local dim=3} follows.

 \strong{Proof that no metric in $\SCYTM$ admits a local LCW}.


It follows from Lemma~\ref{preimage through a transverse map} that the Cotton-York tensor of an SCY-transverse metric is singular in a (possibly empty) compact manifold of dimension $2$, and it never vanishes.
Thus, the subset of $M$ consisting of points whose Cotton-York tensor is not singular is dense, and thus there cannot be any LCW defined in an open set.
The first part of Theorem~\ref{size local dim=3} follows.

%
%
%
%
%
%
%
%

\end{proof}

\end{document}